\documentclass[twoside,leqno,10pt, A4]{amsart}
\usepackage{amsfonts}
\usepackage{amsmath}
\usepackage{amscd}
\usepackage{amssymb}
\usepackage{amsthm}
\usepackage{amsrefs}
\usepackage{latexsym}
\usepackage{mathrsfs}
\usepackage{bbm}
\usepackage{enumerate}
\usepackage{color}
\begin{document}

\newtheorem{theorem}[subsection]{Theorem}
\newtheorem{proposition}[subsection]{Proposition}
\newtheorem{lemma}[subsection]{Lemma}
\newtheorem{corollary}[subsection]{Corollary}
\newtheorem{conjecture}[subsection]{Conjecture}
\newtheorem{prop}[subsection]{Proposition}
\numberwithin{equation}{section}
\newcommand{\mr}{\ensuremath{\mathbb R}}
\newcommand{\mc}{\ensuremath{\mathbb C}}
\newcommand{\dif}{\mathrm{d}}
\newcommand{\intz}{\mathbb{Z}}
\newcommand{\ratq}{\mathbb{Q}}
\newcommand{\natn}{\mathbb{N}}
\newcommand{\comc}{\mathbb{C}}
\newcommand{\rear}{\mathbb{R}}
\newcommand{\prip}{\mathbb{P}}
\newcommand{\uph}{\mathbb{H}}
\newcommand{\fief}{\mathbb{F}}
\newcommand{\majorarc}{\mathfrak{M}}
\newcommand{\minorarc}{\mathfrak{m}}
\newcommand{\sings}{\mathfrak{S}}
\newcommand{\fA}{\ensuremath{\mathfrak A}}
\newcommand{\mn}{\ensuremath{\mathbb N}}
\newcommand{\mq}{\ensuremath{\mathbb Q}}
\newcommand{\half}{\tfrac{1}{2}}
\newcommand{\f}{f\times \chi}
\newcommand{\summ}{\mathop{{\sum}^{\star}}}
\newcommand{\chiq}{\chi \bmod q}
\newcommand{\chidb}{\chi \bmod db}
\newcommand{\chid}{\chi \bmod d}
\newcommand{\sym}{\text{sym}^2}
\newcommand{\hhalf}{\tfrac{1}{2}}
\newcommand{\sumstar}{\sideset{}{^*}\sum}
\newcommand{\sumprime}{\sideset{}{'}\sum}
\newcommand{\sumprimeprime}{\sideset{}{''}\sum}
\newcommand{\shortmod}{\ensuremath{\negthickspace \negthickspace \negthickspace \pmod}}
\newcommand{\V}{V\left(\frac{nm}{q^2}\right)}
\newcommand{\sumi}{\mathop{{\sum}^{\dagger}}}
\newcommand{\mz}{\ensuremath{\mathbb Z}}
\newcommand{\leg}[2]{\left(\frac{#1}{#2}\right)}
\newcommand{\muK}{\mu_{\omega}}

\title[Low-lying zeros of quadratic Hecke $L$-functions of Imaginary Quadratic Number Fields]{One level density of low-lying zeros of
quadratic Hecke $L$-functions of Imaginary Quadratic Number Fields}

\date{\today}
\author{Peng Gao and Liangyi Zhao}

\begin{abstract}
In this paper, we prove a one level density result for the
low-lying zeros of quadratic Hecke $L$-functions of imaginary quadratic number fields of class number one.  As a corollary, we deduce, essentially, that at least $(19-\cot (1/4))/16 = 94.27... \%$ of the $L$-functions under consideration do not vanish at $1/2$.
\end{abstract}

\maketitle



\section{Introduction}

  The theory of random matrices provides fruitful predictions on the behaviors of the zeros of $L$-functions by drawing the analogues between the
  eigenvalues of ensembles of random matrices and the zeros of $L$-functions. In this connection, the density conjecture of N. Katz and P. Sarnak
  \cites{KS1, K&S} asserts that the statistics of the low-lying zeros (zeros near the central point) of various families of $L$-functions are the same as
  that of eigenvalues near $1$ of a corresponding classical compact group. These statistics, the $n$-level densities, have been studied for various
  families of $L$-functions for suitably restricted test functions (see \cites{Gao, O&S, B&F, Ru, Miller1, HuRu, FI, ILS, DuMi2, HuMi, RiRo, Royer, SBLZ1, HB1,
  Brumer, SJM, Young, Gu2, DM}). The results are shown to agree with the density conjecture.
\newline

  To study the $n$-level densities, one needs to form families of $L$-functions and those of $L$-functions associated with quadratic characters are among the most natural ones to consider.   For the family of quadratic Dirichlet $L$-functions, A. E. \"{O}zluk and C. Snyder  \cite{O&S} calculated the one level density and all $n$-level densities were calculated later by M. O. Rubinstein \cite{Ru}. The support of the Fourier transform of the test function in Rubinstein's result is enlarged by the first-named author in \cite{Gao}. It is shown in \cite{ER-G&R, MS1} that these results agree with the density conjecture and the family of quadratic Dirichlet $L$-functions is a symplectic family. Recently, the authors \cite{G&Zhao4} studied another families of quadratic $L$-functions: quadratic Hecke $L$-functions of $\mq(i)$.   Via its one level density computations, it is shown that it is a symplectic family. \newline

Our goal in this paper is to extend the methods used in \cite{G&Zhao4} to study the one level density of the low-lying zeros of quadratic $L$-functions associated with the imaginary quadratic number fields $K=\mq(\sqrt{d})$ of class number 1. Here $d$ is a negative, square-free integer and it is well-known that (see  \cite[(22.77)]{iwakow}) the class number of $K$ is $1$ if and only if $d \in \{ -1,-2,-3,-7,-11,-19,-43,-67,-163 \}$. Since the case $d=-1$ has been studied in \cite{G&Zhao4}, we assume throughout the paper that
\begin{align*}
 d \in \{ -2, -3, -7,-11,-19,-43,-67,-163 \}.
\end{align*}
We note here that if $K$ has class number 1, then all ideals in its ring of integers are principle.  Thus sums over ideals in this setting can be recast as sum of the relevant primary elements (see the discussion in Section 2).  This reparametrization becomes more difficult if the class number is greater than 1, giving a stumbling block to further generalization.   \newline

    The ring of integers $\mathcal{O}_K$ is a free $\mz$ module, $\mathcal{O}_K=\mz+\omega_K \mz$, where
\begin{align*}
   \omega_K & =\begin{cases}
     \frac {1}{2}(1+\sqrt{d}) \qquad & \text{if $d \equiv 1 \pmod 4$ }, \\
     \sqrt{d} \qquad & \text{if $d \equiv 2, 3 \pmod 4$ }.
    \end{cases}
\end{align*}

Let
\[ C_K(X) = \{ c \in \mathcal{O}_K : (c,2) =1, \; c \; \mbox{square-free}, \; X \leq N(c) \leq 2X \} . \]
 We define $c_K=\sqrt{-2}$ when $d=-2$ and $c_K=2$ otherwise. We shall define in Section \ref{sec2.4} the primitive quadratic Kroncecker symbol
 $\chi^{(-4c_Kc)}$ for $c \in C_K(X)$.
We denote the non-trivial zeros of the Hecke $L$-function
   $L(s, \chi^{(-4c_Kc)})$ by $\half+i \gamma_{\chi^{(-4c_Kc)}, j}$.  Without assuming the generalized Riemann hypothesis (GRH), we order them as
\begin{equation*}
    \ldots \leq
   \Re \gamma_{\chi^{(-4c_Kc)}, -2} \leq
   \Re \gamma_{\chi^{(-4c_Kc)}, -1} < 0 \leq \Re \gamma_{\chi^{(-4c_Kc)}, 1} \leq \Re \gamma_{\chi^{(-4c_Kc)}, 2} \leq
   \ldots.
\end{equation*}
    Moreover, we normalize the $\gamma$'s by setting
\begin{align*}
    \widetilde{\gamma}_{\chi^{(-4c_Kc)}, j}= \frac{\gamma_{\chi^{(-4c_Kc)}, j}}{2 \pi} \log X.
\end{align*}
This is done so that the average spacing of the zeros is 1.  Moreover, we define, for an even Schwartz class function $\phi$,
\begin{equation*} 
S_K(\chi^{(-4c_Kc)}, \phi)=\sum_{j} \phi(\tilde{\gamma}_{\chi^{(-4c_Kc)}, j}).
\end{equation*}

We further let $\Phi_X(t)$ be a non-negative smooth function supported on $(1,2)$,
    satisfying $\Phi_X(t)=1$ for $t \in (1+1/U, 2-1/U)$ with $U=\log \log X$ and such that
    $\Phi^{(j)}_X(t) \ll_j U^j$ for all integers $j \geq 0$.  We define the one level density associated to the quadratic family of $\{ L(s,
    \chi^{(-4c_Kc)}) \}$ with $c \in C_K(X)$ as
\begin{align*}
   \frac{1}{\# C_K(X)}\sumstar_{(c, 2)=1}  S_K(\chi^{(-4c_Kc)}, \phi).
\end{align*}
This is a weighted count of the low-lying zeros in our family. \newline

    Our result is
\begin{theorem}
\label{quadraticmainthm}
Suppose that GRH is true.  Let $\phi(x)$ be an even Schwartz function whose
Fourier transform $\hat{\phi}(u)$ has compact support in $(-2,2)$, then
\begin{align}
\label{quaddensity}
 \lim_{X \rightarrow +\infty}\frac{1}{\# C_K(X)}\sumstar_{(c, 2)=1}  S_K(\chi^{(-4c_Kc)}, \phi)\Phi_X \left( \frac {N(c)}{X} \right)
 = \int\limits_{\mathbb{R}} \phi(x) W_{USp}(x) \dif x, \quad \mbox{where} \; \; W_{USp}(x)=1-\frac{\sin(2\pi x)}{2\pi x}.
\end{align}
   Here the ``$*$'' on the sum over $c$ means that the sum is restricted to square-free elements $c$ of $\mathcal{O}_K $.
\end{theorem}

The kernel of the integral $W_{USp}$ in \eqref{quaddensity} is the same function which occurs in the one level density of the eigenvalues of unitary symplectic matrices. This shows that the family of quadratic Hecke $L$-functions of $\mathcal{O}_K $ is a symplectic family.  We remark here that a kind of repulsion of zeros from the central point has been observed in symplectic families, that is, the probability of finding an $L$-function in a symplectic family with a low first zero above the central point is very small.  See, for example, \cites{Mil, Mar}, for detailed discussions of the phenomenon.  \newline

  Our proof of Theorem \ref{quadraticmainthm} proceeds along the same line of arguments as those in \cite{G&Zhao4}. The allowable range of the support of the Fourier transform of the test function in \eqref{quaddensity} depends on certain character sums. To get a better estimation, we apply a   two dimensional Poisson summation over $\mathcal{O}_K$  to convert long character sums to shorter ones, which leads to a second main term in the computation of the one level density. In this process, we need to evaluate explicitly certain Gauss sums at prime arguments, which we achieve by combining the quadratic reciprocity laws for imaginary quadratic number fields and the approaches used in \cite{B&S, A&A&W}. \newline
  
A conjecture that goes back to S. Chowla \cite{chow} asserts that $L(1/2, \chi) \neq 0$ for any primitive quadratic Dirichlet character $\chi$.  In general, it is believed that an $L$-function does not vanish at its central points unless there is a compelling reason (the sign of the functional equation is $-1$ or the $L$-function is associated with an elliptic curve of positive rank, a connection given by the Birch--Swinnerton-Dyer conjecture).  Thus, it stands to reason that the same statement should hold for Hecke characters as well.  Motivated by this conjecture, we use Theorem~\ref{quadraticmainthm} to deduce the following
  
  \begin{corollary}
 Suppose that the GRH is true and that $1/2$ is a zero of $L \left( s, \chi^{(-4c_Kc)} \right)$ of order $m_c \geq 0$.  As $X \to \infty$,
 \begin{equation} \label{coreq1}
  \sumstar_{(c, 2)=1}  m_c \Phi_X \left( \frac {N(c)}{X} \right) \leq  \left( \frac{\cot \frac{1}{4} -3 }{8} + o(1) \right) \# C_K(X).
  \end{equation}
   Moreover, as $X \to \infty$
 \begin{equation} \label{coreq2}
  \# \{ c \in C_K(X) : L\left( 1/2, \chi^{(-4c_Kc)} \right) \neq 0 \} \geq \left( \frac{19-\cot \frac{1}{4}}{16} + o(1) \right) \# C_K(X) \geq 0.9427 \cdot \# C_K(X) .
  \end{equation}
    \end{corollary}

\begin{proof}
The proof is similar to that of \cite[Corollary 2.1]{B&F} and goes along the same line as the optimization carried out in \cite[Appendix A]{ILS}.  Suppose that $\phi$ satisfies the conditions of Theorem~\ref{quadraticmainthm} and $\phi(0)=1$.   Discarding all terms in $S_K(\chi^{(-4c_Kc)}, \phi)$ with $\tilde{\gamma}_{\chi^{(-4c_Kc)}, j} \neq 0$, we get
\begin{equation} \label{opt}
  \frac{1}{\# C_K(X)} \sumstar_{(c, 2)=1}  m_c \Phi_X \left( \frac {N(c)}{X} \right) \leq \int\limits_{\rear} \phi(x) W_{USp}(x) \dif x.
  \end{equation}
It thus remains the minimize the right-hand side of the above expression, by taking the infimum over all admissible $\phi$. \newline

Following the computations carried out in \cites{B&F, ILS}, we have
\begin{equation} \label{infres}
\inf_{\phi \in \mathfrak{F}} \int\limits_{\rear} \phi(x) W_{USp}(x) \dif x = \frac{\cot \frac{1}{4} -3 }{8} .
\end{equation}
Here $\mathfrak{F}$ consists of even Schwartz functions $\phi \in L^1(\rear)$ with $\phi(0)=1$ and Fourier transform supported in $(-2,2)$.   The infimum is attain by the function $\phi$ whose Fourier transform $\hat{\phi}$ is
\[ \hat{\phi}(y) = \left( h \star \check{h} \right) (y) , \]
Here $f \star g$ denotes the convolution of $f$ and $g$, $h(y)$ is the even extension of
\[ h(y) = \left\{ \begin{array}{cc} \displaystyle{\frac{\sin \left( \frac{y}{2} - \frac{\pi +1}{4} \right)}{\sqrt{2} \sin \frac{1}{4} - \cos \frac{\pi+1}{4} } }&  \quad 0 \leq y \leq 1 \\ \\ 0 & y > 0 \end{array} \right. \]
and $\check{h}(y) = \overline{h(-y)}$. \newline

The infimum of the right-hand side of \eqref{opt} taken over all admissible $\phi$'s under the conditions of Theorem~\ref{quadraticmainthm} does not exceed the quantity in \eqref{infres}.  This proves \eqref{coreq1}. \newline

Moreover, as we are dealing with a family of quadratic characters and hence the root numbers of these $L$-functions are 1, $m_c$ must be even and thus at least 2 if positive.  From this observation and \eqref{coreq1}, we arrive at \eqref{coreq2}.
\end{proof}

\subsection{Notations} The following notations and conventions are used throughout the paper.\\
\noindent We write $\Phi(t)$ for $\Phi_X(t)$. \newline
\noindent $e(z) = \exp (2 \pi i z) = e^{2 \pi i z}$. \newline
$f =O(g)$ or $f \ll g$ means $|f| \leq cg$ for some unspecified
positive constant $c$. \newline
$f =o(g)$ means $\displaystyle \lim_{x \rightarrow \infty}f(x)/g(x)=0$. \newline
$\mu_{K}$ denotes the M\"obius function on $\mathcal{O}_K $. \newline
$\chi_{[-1,1]}$ denotes the characteristic function of $[-1,1]$.

\section{Preliminaries}
\label{sec 2}

\subsection{Imaginary quadratic number fields of class number $1$}
\label{sect: Kronecker}

    The following facts concerning $K=\mq(\sqrt{d})$ can be found in \cite[Section 3.8]{iwakow}. The group of units $U_k=\{ \pm 1 \}$ based on our assumption on $d$.  The discriminant $D_K$ of $K$ is
\begin{align*}
   D_K & =\begin{cases}
     d \qquad & \text{if $d \equiv 1 \pmod 4$ }, \\
     4d \qquad & \text{if $d \equiv 2, 3 \pmod 4$ }.
    \end{cases}
\end{align*}

    The different of $K$ is the pricipal ideal $(\delta_K)=(\sqrt{D_K})$. The Kronecker symbol in $\mz$, $\chi_{D_K}(\cdot)=\leg {D_k}{\cdot}_{\mz} $ is a
    real primitive character of conductor $-D_K$ and the law of factorization of rational primes $p$ into prime ideals in $K$ asserts the following:
\begin{align*}
    & \chi_{D_K}(p)=0, \text{then $p$ ramifies}, (p)=\mathfrak{p}^2 \quad \text{with} \quad N(\mathfrak{p})=p, \\
    & \chi_{D_K}(p)=1, \text{then $p$ splits}, (p)=\mathfrak{p}\overline{\mathfrak{p}} \quad \text{with} \quad  \mathfrak{p} \neq \overline{\mathfrak{p}}
    \quad \text{and} \quad  N(\mathfrak{p})=p, \\
    & \chi_{D_K}(p)=-1, \text{then $p$ is inert}, (p)=\mathfrak{p} \quad \text{with} \quad  N(\mathfrak{p})=p^2.
\end{align*}

   In particular, the rational
    prime $2$ splits when $d \equiv 1 \pmod 8$, is inert when $d \equiv 5 \pmod 8$ and ramifies when  $d=-2$.

\subsection{Quadratic characters and Kronecker symbols}
\label{sec2.4}

    The symbol $(\frac{\cdot}{n})$ is the quadratic
residue symbol in the ring of integers $\mathcal{O}_K $.  For a prime $\varpi \in \mathcal{O}_K, (\varpi, 2)=1$, the quadratic character is defined for $a \in
\mathcal{O}_K$, $(a, \varpi)=1$ by $\leg{a}{\varpi} \equiv
a^{(N(\varpi)-1)/2} \pmod{\varpi}$, with $\leg{a}{\varpi} \in \{
\pm 1 \}$. When $\varpi | a$, we define
$\leg{a}{\varpi} =0$.  Then the quadratic character can be extended
to any composite $n$ with $(N(n), 2)=1$ multiplicatively. We further define $\leg {\cdot}{n}=1$ when $n$ is a unit in $\mathcal{O}_K$. \newline

    For any $n \in \mathcal{O}_K, (n, 2)=1$, let $n=a+b\omega_K$ and $n \omega_K=c+d\omega_K$ with $a,b,c,d \in \mz$. Then $\chi(n)=i^w$, where
  $w=(b^2-a+2)c+(a^2-b+2)d+ad$,
defines a map defined modulo $4\mathcal{O}_K$, and we have the following quadratic reciprocity law for quadratic imaginary number fields \cite[Theorem 8.15]{Lemmermeyer} that for $(mn,2)=1$,
\begin{align*}
    \leg {n}{m}\leg{m}{n}=(-1)^{((N(n)-1)/2)((N(m)-1)/2)}\chi(m)^{N(m)((N(n)-1)/2)}\chi(n)^{N(n)((N(m)-1)/2)}.
\end{align*}

    It follows from the above quadratic reciprocity law that when $(mn,2)=1$ and $m \equiv 1 \pmod 4$, then (see \cite[Propostion 8.17]{Lemmermeyer})
\begin{align}
\label{quadreciKm1}
    \leg {n}{m}\leg{m}{n}=1.
\end{align}

  One also has the following supplementary laws (see \cite[Propostion 4.2(ii), (iii)]{Lemmermeyer}),
\begin{align}
\label{2.05}
  \leg {-1}{n}=(-1)^{(N(n)-1)/2}, \hspace{0.1in} \leg {2}{n}=\leg {2}{N(n)}_{\mz},
\end{align}
   where $\leg {\cdot}{\cdot}_{\mz}$ denotes the Jacobi symbol in $\mz$. Moreover, it follows from \cite[Propostion 4.2(iii)]{Lemmermeyer}
that for $a \in \mz,  m \in \mathcal{O}_K, (a,2)=1$, we have
\begin{align}
\label{2.06}
   \leg {m}{a}=\leg {N(m)}{a}_{\mz}.
\end{align}

We deduce from \eqref{2.05} that when $d \neq -2$, for $n \equiv 1 \pmod {4c_Kc}$,
\begin{align} \label{2.77}
  \leg {4c_K}{n}=1.
\end{align}

Now assume $d=-2$ and $n \equiv 1 \pmod {4c_Kc}$.  In this case, we have $c_K=\omega_K$. We write $n=a+b\omega_K$ with $a,b\in \mz$ to see that $a \equiv 1 \pmod 8$ and we can write $b=2^l b'$ with $(b',2)=1,  b', l \in \mz, l \geq 0$.  We consider
\begin{align}
\label{2.6}
  \leg {a}{a+b\omega_K}=\leg {-b\omega_K}{a+b\omega_K}=\leg {-b}{a+b\omega_K}\leg {c_K}{a+b\omega_K}.
\end{align}

  On the other hand, as $a \equiv 1 \pmod 8$, it follows from \eqref{2.06} that when $a,b \in \mz, a \equiv 1 \pmod 8$, then
\begin{align*}
  \leg {b}{a}=1, \quad \leg {\omega_K}{a}=\leg {N(\omega_K)}{a}_{\mz}=\leg {2}{a}_{\mz}=1.
\end{align*}

  We deduce from the above and \eqref{quadreciKm1} that
\begin{align}
\label{2.10}
  \leg {a}{a+b\omega_K}=\leg {a+b\omega_K}{a}=\leg {b\omega_K}{a}=\leg {b}{a}\leg {\omega_K}{a}=1, \\
  \leg {-b}{a+b\omega_K}=\leg {-2^lb'}{a+b\omega_K}=\leg {2^l}{a+b\omega_K}\leg {a+b\omega_K}{-b'}=\leg {a}{-b'}=1. \nonumber
\end{align}

Combining \eqref{2.6} and \eqref{2.10}, we see that \eqref{2.77} still holds for $n \equiv 1 \pmod {4c_Kc}$ when $d=-2$.
The above discussions allow us to define a quadratic Dirichlet character $\chi^{(-4c_Kc)} \pmod {4c_Kc}$ for any element $c \in \mathcal{O}_K, (c, 2)=1$, such that for any $n \in (\mathcal{O}_K/(4c_Kc\mathcal{O}_K))^*$,
\begin{align*}
   \chi^{(-4c_Kc)}(n)=\leg {-4c_Kc}{n}.
\end{align*}

    One deduces from \eqref{2.77} and the quadratic reciprocity that  $\chi^{(-4c_Kc)}(n)=1$ when $n \equiv 1 \pmod {4c_Kc}$. It follows from this that
    $\chi^{(-4c_Kc)}(n)$ is well-defined . As $\chi^{(-4c_Kc)}(n)$ is clearly multiplicative and of order $2$ and is trivial on units, it can be regarded as a
    quadratic Hecke character $\pmod {4c_Kc}$ of trivial infinite type. We denote $\chi^{(-4c_Kc)}$ for this Hecke character as well and we call it the
    Kronecker symbol. Furthermore, when $c$ is square-free, $\chi^{(-4c_Kc)}$ is non-principal and primitive. To see this, we write $c=u_c \cdot \varpi_1
    \cdots \varpi_k$ with $u_c$ a unit and $\varpi_j$  being primes. Suppose $\chi^{(-4c_Kc)}$ is induced by some $\chi$ modulo $c'$ with $\varpi_j \nmid
    c'$, then by the Chinese Remainder Theorem, there exists an $n$ such that $n \equiv 1 \pmod {4c_Kc/\varpi_j}$ and $\leg {\varpi_j}{n} \neq 1$. It
    follows from this and \eqref{quadreciKm1} that $\chi(n)=1$ but $\chi^{(-4c_Kc)}(n) \neq 1$, a contradiction. Thus, $\chi^{(-4c_Kc)}$ can only be possibly
    induced by some $\chi$ modulo $4a_Kc$, where $a_K=\omega_K$ or $\overline{\omega_K}$ when $d \equiv 1 \pmod 8$ (note that in this case $d=-7$ and $2=\omega_K \overline{\omega_K}$) and $a_K=1$ otherwise. Suppose it is induced by some $\chi$
    modulo $4a_Kc$, then by the Chinese Remainder Theorem, there exists an $n$ such that $n \equiv 1 \pmod {c}$ and $n \equiv 1+4b_K \pmod {8}$, where
    $b_K=\omega_K$ when $d \equiv 5 \pmod 8$ and $b_K=a_K$ otherwise.  When $d=-2$,  we can further take this $n$ to satisfy $n \in \mz$ (for example, we
    can take any $n \in \mz$ satisfying $n \equiv 1 \pmod {N(c)}$ and $n \equiv 5 \pmod 8$). As this $n \equiv 1 \pmod {4a_K}$, it follows that $n \equiv 1
    \pmod {4a_Kc}$, hence $\chi(n)=1$ but $\chi^{(-4c_Kc)}(n)=\leg {c_K}{n}=-1$ (note that $\leg {u}{n}=1$ when $u$ is a unit in $\mathcal{O}_K$),
    where the last equality follows from \eqref{2.05} when $d \neq -2$ and \eqref{2.06} when $d=-2$.
    This implies that $\chi^{(-4c_Kc)}$ is primitive. This also shows that $\chi^{(-4c_Kc)}$ is non-principal.

\subsection{The Gauss sums}
     For any $n \in  \mathcal{O}_K$, $(n,2)=1$, the quadratic Gauss sum $g_K(n)$ is defined by
\begin{align*}
   g_K (n) = \sum_{x \bmod{n}} \leg{x}{n} \widetilde{e}_K\leg{x}{n}, \quad \mbox{where} \quad \widetilde{e}_K(z) =\exp \left( 2\pi i  \left( \frac {z}{\sqrt{D_K}} - \frac {\overline{z}}{\sqrt{D_K}} \right) \right) .
   \end{align*}

   The following well-known relation (see \cite[p. 195]{P}) holds for all $n$:
\begin{align*} 
   \left| g_K(n) \right| & =\begin{cases}
    \sqrt{N(n)} \qquad & \text{if $n$ is square-free}, \\
     0 \qquad & \text{otherwise}.
    \end{cases}
\end{align*}

From the definition of $g_K(n)$, one easily derives the relation
\begin{align*} 
   g_K(n_1 n_2) =\leg{n_2}{n_1}\leg{n_1}{n_2}g_K(n_1) g_K(n_2), \quad (n_1, n_2) = 1.
\end{align*}

   For primes in $\mathcal{O}_K$ that are co-prime to $2D_K$ and $(1-d)/4$ when $d \equiv 1 \pmod 4$, we now introduce the notion of primary. When a prime is also a rational prime $p$, then we say it is primary when $p>0$. When a prime $\varpi$ satisfies $N(\varpi)=p$ with $p$ a rational
   prime, then we write $\varpi=a+b\omega_K$ to see that
\begin{align} \label{p}
   p & =\begin{cases}
    \displaystyle a^2+ab+b^2\frac {1-d}{4} \qquad & \text{if $d \equiv 1 \pmod 4$ }, \\
     \displaystyle a^2-db^2  \qquad & \text{if $d \equiv 2, 3 \pmod 4$ }.
    \end{cases}
\end{align}
    Note that $p$ is also co-prime to $2D_k$ and $(1-d)/4$ when $d \equiv 1 \pmod 4$.  One deduces easily from this and the above expression for $p$ that $(ab,p)=1$. \newline

    We first consider the case when $d \equiv 1 \pmod 4$. In this case, we have
\begin{align*}
   \leg {-b(a+b(1-d)/4)}{p}_{\mz}= \leg {a^2}{p}_{\mz}=1.
\end{align*}

    It follows that
\begin{align*}
   \leg {b}{p}_{\mz}= \leg {-(a+b(1-d)/4)}{p}_{\mz}.
\end{align*}

    We now define the prime $\varpi=a+b\omega_K$ to be primary if $b \equiv 1 \pmod 4$ or $a+b(1-d)/4 \equiv -1 \pmod 4$ when $(b, 2) \neq 1$. Note that
    when $(b, 2) \neq 1$, we must have $(a,2)=1$, for otherwise we will have $2|p$, contradicting the fact that $p$ is co-prime to $2$. Since $U_K=\{ \pm 1
    \}$, it is easy to see that each prime ideal in $K$ has a unique primary generator $\varpi$. When $\varpi$ is primary, we deduce via quadratic
    reciprocity and \eqref{p} that
\begin{align}
\label{2.8}
   \leg {b}{p}_{\mz}= \leg {-(a+b(1-d)/4)}{p}_{\mz}=1.
\end{align}

    Now suppose that $d=-2$. In this case,  we define the prime $\varpi=a+b\omega_K$ to be primary if $b=2^kb'$ with $k, b' \in \mz, k \geq 0, b' \equiv 1 \pmod 4$. Note that when
    $k \geq 1$, then again $(a,2)=1$ and it follows from \eqref{p} that $p \equiv 1 \pmod 8$, hence $\leg {2}{p}_{\mz}=1$.  Again each prime ideal in $K$ has a
    unique primary generator $\varpi$ and when $\varpi$ is primary, we deduce via quadratic reciprocity and \eqref{p} that
\begin{align}
\label{2.8'}
   \leg {b}{p}_{\mz}= 1.
\end{align}

  We now evaluate the Gauss sum at each primary prime $\varpi$. We have the following
\begin{lemma}
\label{Gausssum}
   Let $\varpi$ be a primary prime in $\mathcal{O}_K$. Then
\begin{align*}
   g_K(\varpi)= \begin{cases}
    N(\varpi)^{1/2} \qquad & \text{if} \qquad N(\varpi) \equiv 1 \pmod 4,\\
    -i N(\varpi)^{1/2} \qquad & \text{if} \qquad N(\varpi) \equiv -1 \pmod 4.
\end{cases}
\end{align*}
\end{lemma}
\begin{proof}
First, assume that $p=N(\varpi)$ is a rational prime so that $p=N(\varpi)=\varpi \overline{\varpi}$ with $(\varpi, \overline{\varpi})=1$ (note that by our definition of primary, $(p)$ does not ramify in $\mathcal{O}_K$. We use $R$ to denote a quadratic residue $\pmod \varpi$ and $N$
    to denote a quadratic non-residue $\pmod \varpi$. Note first that
\begin{align*}
  E_K :=\sum_{x \bmod \varpi}\tilde{e}_K \left( \frac
  x \varpi \right)=1+ \sum_{R \bmod{\varpi}} \widetilde{e}_K\leg{R}{\varpi}+\sum_{N \bmod{\varpi}} \widetilde{e}_K\leg{N}{\varpi}.
\end{align*}
   As $(\varpi, \overline{\varpi})=1$, we get that
\begin{align*}
  \sum_{x \bmod \varpi}\tilde{e}_K \left( \frac x \varpi \right)=\sum_{x \bmod \varpi}\widetilde{e}_K \left( \frac { x\overline{\varpi}}{ \varpi \overline{\varpi}} \right)=\sum_{x
  \bmod \varpi}\widetilde{e}_K \left( \frac { x\overline{\varpi}}{p} \right)=\sum_{x \bmod \varpi}\widetilde{e}_K \left( \frac {x}{p} \right).
\end{align*}
   On summing over a set of representatives $\pmod p$ instead of $\pmod \varpi$ and noting that $\omega_K-\overline{\omega}_K=\sqrt{D_K}$, we deduce that
\begin{align*}
  E_K \frac {N(p)}{N(\varpi)}=\sum_{x \bmod p}\widetilde{e}_K \left( \frac {x}{p} \right)=\sum_{y,z \in \mz / p\mz }\widetilde{e}_K \left( \frac {y+z\omega_K}{p} \right)=\sum_{z \in   \mz / p\mz }e \left( \frac {z}{p} \right) =0.
\end{align*}

   We conclude that $E_K=0$ and it follows that
\[ g_K (\varpi) = \sum_{R \bmod{\varpi}} \widetilde{e}_K\leg{R}{\varpi}-\sum_{N \bmod{\varpi}} \widetilde{e}_K\leg{N}{\varpi} =1+2\sum_{R \bmod{\varpi}} \widetilde{e}_K\leg{R}{\varpi} =\sum_{\nu \bmod{\varpi}} \widetilde{e}_K\leg{\nu^2}{\varpi}. \]

Now $g_K(\varpi)$ can be further rewritten as
\begin{align*}
   g_K(\varpi) =\sum_{\nu \bmod{\varpi}} \widetilde{e}_K\leg{\nu^2 \overline{\varpi}}{\varpi \overline{\varpi}}=\sum_{\nu \bmod{\varpi}}
   \widetilde{e}_K\leg{\nu^2 \overline{\varpi}}{p}.
\end{align*}

    On setting $\nu=x+y\omega_K$, $\varpi=a+b\omega_K$ so that $\overline{\varpi}=a+b\overline{\varpi}_K$,  we sum over a set of representatives $\pmod p$
    instead of $\pmod \varpi$ to deduce that
\begin{align}
\label{g_K}
   g_K(\varpi)\frac {N(p)}{N(\varpi)} =\sum_{\nu
   \bmod{p}} \widetilde{e}_K\leg{\nu^2 \overline{\varpi}}{p}
   =\begin{cases}
    \displaystyle \sum_{x,y \in \mz / p\mz }e \left( \frac {-bx^2+2axy-bdy^2}{p} \right) \qquad & \text{if $d \equiv 2, 3 \pmod 4$ }, \\ \\
    \displaystyle \sum_{x,y \in \mz / p\mz }e \left( \frac {-bx^2+2axy+(a+b(1-d)/4)y^2}{p} \right)  \qquad & \text{if $d \equiv 1 \pmod 4$ }.
    \end{cases}
\end{align}

    We are thus led to consider the evaluation of the following exponential sum:
\begin{align*}
    \sum_{x,y \in \mz / M\mz }e \left( \frac {fx^2+gxy+hy^2}{M} \right),
\end{align*}
   where $M,f,g,h \in \mz$ with $2|g, 2 \nmid M$. \newline

    Let $A$ be any rational integer co-prime to $M$ which is represented by the
binary quadratic form $fx^2+gxy+hy^2$. Then there exist integers $r$ and $t$ such that
\begin{align*}
A = fr^2 + grt+ ht^2 \not \equiv 0 \pmod M.
\end{align*}
    Define rational integers $s$ and $u$ by
\begin{align*}
 s=-gr-2ht, \quad u=2fr+gt.
\end{align*}

   Then we have
\begin{align*}
 \left (
\begin{matrix}
r& s\\
t & u
\end{matrix}
\right )^T\left (
\begin{matrix}
f& g/2\\
g/2 & h
\end{matrix}
\right )\left (
\begin{matrix}
r& s\\
t & u
\end{matrix}
\right )=\left (
\begin{matrix}
A & 0\\
0 & (4fh-g^2)A
\end{matrix}
\right )
\end{align*}
  in $\text{\it GL}(2, \mz/M\mz)$. \newline

  Moreover, we have
\begin{align*}
  ru-st =2A \not \equiv 0 \pmod M.
\end{align*}
   Thus, we have
\begin{align*}
 \left (
\begin{matrix}
r& s\\
t & u
\end{matrix}
\right ) \in \text{\it GL}(2, \mz/M\mz).
\end{align*}
  The above matrix hence induces an invertible change of variables so that we have
\begin{align*}
  \sum_{x,y \in \mz / M\mz }e \left( \frac {fx^2+gxy+hy^2}{M} \right)=\sum_{x,y \in \mz / M\mz }e \left( \frac {Ax^2+(4fh-g^2)Ay^2}{M} \right).
\end{align*}

   We can now apply the discussions above to study the expression for $g_K(\varpi)$ in \eqref{g_K}. Note that by taking $x=1, y=0$, we see that we can take $A$ to be $-b$ in our case and in
   either case we have
\begin{align*}
   p g_K(\varpi) =\sum_{x,y \in \mz / p\mz }e \left( \frac {-bx^2-4pby^2}{p} \right)=p\sum_{x\in \mz / p\mz }e \left( \frac {-bx^2}{p} \right)=p\leg {-b}{p}_{\mz}\tau(\chi_p),
\end{align*}
   where
\begin{align*}
   \tau(\chi_p)=\sum_{x\in \mz / p\mz }e \left( \frac {x^2}{p} \right)
\end{align*}
   is the quadratic Gauss sum in $\mq$. It follows from \eqref{2.8} and \eqref{2.8'} that
\begin{align*}
   g_K(\varpi) =\leg {-1}{p}\tau(\chi_p).
\end{align*}
   As it follows from \cite[Chap. 2]{Da} that $\tau(\chi_p)=p^{1/2}=N(\varpi)^{1/2}$ when $p \equiv 1 \pmod 4$ and $\tau(\chi_p)=ip^{1/2}=iN(\varpi)^{1/2}$
   when $p \equiv 3 \pmod 4$. \newline
   
Now if $\varpi=p$ is an odd rational prime, we compute $g(p)$ directly. As in \cite[Chap. 2]{Da} (note that in this case, we still have $\sum_{x \bmod p}\tilde{e}(x/p)=0$), we have
\begin{equation*}
  g(p) =  \sum_{x \bmod{p}}\tilde{e} \left( \frac {x^2}p \right).
\end{equation*}
   We now write $x=a+b\omega_K$ with $a, b \pmod p$ in $\mz$ to see that
\begin{equation*}
  g(p) =  \sum_{x \bmod{p}}\tilde{e} \left( \frac {x^2}p \right)=\sum^p_{b=1}\sum^p_{a=1}e \left( \frac {2ab-b^2}{p} \right)=p=N(p)^{1/2}.
\end{equation*}
Recall that $N(p)=p^2 \equiv 1 \pmod 4$.  This completes the proof of the lemma.
\end{proof}

     Now, we define more generally, for any $n,k \in \mathcal{O}_K, (n,2)=1$,
\begin{align*}
 g_K(k,n) = \sum_{x \bmod{n}} \leg{x}{n} \widetilde{e}\leg{kx}{n} \qquad \mbox{and} \qquad G_K(k,n) = \left (\frac {1-i}{2}+\leg {-1}{n}\frac {1+i}{2}
 \right )g_K(k,n).
\end{align*}

    Note that
\begin{align}
\label{gtoG}
 g_K(k,n) = \left (\frac {1+i}{2}+\leg {-1}{n}\frac {1-i}{2} \right )G_K(k,n).
\end{align}

    The notion of $G_K(k,n)$ is motivated by the work of K. Soundararajan \cite[Section 2.2]{sound1}.
    For our purpose in this paper, we only need know the explicit values of $G_K(k,n)$ at primary primes $n$, which is given in the following lemma, whose proof is omitted as it is similar to that of \cite[Lemma 2.4]{G&Zhao4}.
\begin{lemma} \label{GausssumMult}
(i) We have
\begin{align}
\label{2.7}
   G_K(rs,n) & = \overline{\leg{s}{n}} G_K(r,n), \qquad (s,n)=1.
\end{align}
(ii) Let $\varpi$ be a primary prime in $\mathcal{O}_K$. Suppose $\varpi^{h}$ is the largest power of $\varpi$ dividing $k$. (If $k = 0$ then set $h =
\infty$.) Then for $l \geq 1$,
\begin{align*}
 G_K(k, \varpi^l)& =\begin{cases}
    0 \qquad & \text{if} \qquad l \leq h \qquad \text{is odd},\\
    \varphi(\varpi^l)=\#(\mathcal{O}_K/(\varpi^l))^* \qquad & \text{if} \qquad l \leq h \qquad \text{is even},\\
    -N(\varpi)^{l-1} & \text{if} \qquad l= h+1 \qquad \text{is even},\\
    \leg {-k\varpi^{-h}}{\varpi}N(\varpi)^{l-1/2} \qquad & \text{if} \qquad l= h+1 \qquad \text{is odd},\\
    0 \qquad & \text{if} \qquad l \geq h+2.
\end{cases}
\end{align*}
\end{lemma}

\subsection{The Explicit Formula}
\label{section: Explicit Formula}

Our approach in this paper relies on the following explicit
formula, which essentially converts a sum over zeros of an
$L$-function to a sum over primes. As it is similarly to that of \cite[Lemma 2.7]{G&Zhao4}, we omit the proof here.
\begin{lemma}
\label{lem2.1}
   Let $\phi(x)$ be an even Schwartz function whose Fourier transform
   $\hat{\phi}(u)$ is compactly supported. For any square-free $c \in \mathcal{O}_K, (c, 2)=1, X \leq N(c) \leq 2X$, we have
\begin{align*}
  S \left( \chi^{(-4c_Kc)}, \phi \right)   =\int\limits^{\infty}_{-\infty}\phi(t) \dif t-\frac 1{2}
   \int\limits^{\infty}_{-\infty}\hat{\phi}(u) \dif u-2S \left( \chi^{(-4c_Kc)},X; \hat{\phi} \right)+O \left( \frac {\log \log 3X}{\log
   X} \right),
\end{align*}
     with the implicit constant depending on $\phi$, where
\begin{equation*}
   S \left( \chi^{(-4c_Kc)},X;\hat{\phi} \right)=\frac 1{\log X}\sum_{\varpi \text{ primary}}\frac {\log
   N(\varpi)}{\sqrt{N(\varpi)}}\chi^{(-4c_Kc)}(\varpi)  \hat{\phi}\left( \frac {\log N(\varpi)}{\log X} \right).
\end{equation*}
\end{lemma}

   We note that in the expression of $S(\chi^{(-4c_Kc)},X;\hat{\phi})$, the sum is only over primary primes as there are only finitely many prime ideals that have no primary prime generators and hence they only contribute to the error term. \newline

     We shall also need the following
\begin{lemma}
\label{lem2.7}
Suppose that GRH is true. For any non-principal Hecke character $\chi$ of trivial infinite type with modulus $n$, we have for $x \geq 1$,
\begin{align*}
  S(x, \chi)=\sum_{\substack {N(\varpi) \leq x \\ \varpi \text{ primary}}} \chi (\varpi) \log N(\varpi)
\ll \min \left\{ x, x^{1/2}\log^{3} (x) \log N(n) \right\}.
\end{align*}
\end{lemma}

\begin{proof}
The $x$ in the minimum is trivial and the other term comes from the standard analysis of $L'(s,\chi)/L(s,\chi)$.  This Lemma is the same as \cite[Lemma 2.5]{G&Zhao4} and the proof is given there.
\end{proof}

\subsection{Poisson Summation}
      Our proof of Theorems \ref{quadraticmainthm} requires the following application of Poisson summation:
\begin{lemma}
\label{Poissonsum} Let $n \in \mathcal{O}_K, (n, 2)=1$ and and $\leg {\cdot}{n}$ be the quadratic symbol $\pmod {n}$. For any Schwartz class function $W$,  we have for all
$a>0$,
\begin{align}
\label{PoissonsumQw}
   \sum_{m \in \mathcal{O}_K}\leg {m}{n}W\left(\frac {aN(m)}{X}\right)=\frac {X}{aN(n)}\sum_{k \in
   \mathcal{O}_K}G_K(k,n)\widetilde{W}_K\left(\sqrt{\frac {N(k)X}{aN(n)}}\right),
\end{align}
   where
\begin{align*}
   \widetilde{W}_K(t) &=\int\limits^{\infty}_{-\infty}\int\limits^{\infty}_{-\infty}W(N(x+y\omega_K))\widetilde{e}\left(- t(x+y\omega_K)\right)\dif x \dif y, \quad t
   \geq 0.
\end{align*}
\end{lemma}
\begin{proof}
   We first recall the following Poisson summation formula for
   $\mathcal{O}_K$ (see the proof of \cite[Lemma 4.1]{G&Zhao}), which is itself an easy consequence of the classical Poisson summation formula in
$2$ dimensions:
\begin{align*}
   \sum_{j \in \mathcal{O}_K}f(j)=\sum_{k \in
   \mathcal{O}_K}\int\limits^{\infty}_{-\infty}\int\limits^{\infty}_{-\infty}f(x+y\omega_K)\widetilde{e}\left( -k(x+y\omega_K) \right)\dif x \dif y.
\end{align*}
   We then have
\begin{align*}
    \sum_{m \in \mathcal{O}_K}\leg {m}{n}W\left(\frac {aN(m)}{X}\right) =& \sum_{r \shortmod n}\leg {r}{n}\sum_{j \in \mathcal{O}_K}W\left(\frac
    {aN(r+jn)}{X}\right)  \\
   =& \sum_{r \shortmod n}\leg {r}{n} \sum_{k \in
   \mathcal{O}_K}\int\limits^{\infty}_{-\infty}\int\limits^{\infty}_{-\infty}W\left(\frac {aN(r+(x+y\omega_K)n)}{X}\right)\widetilde{e}\left(-k(x+y\omega_K) \right) \dif
   x \dif y.
\end{align*}
   We change variables in the integral, writing
\begin{equation*}
   \sqrt{N\Big(\frac {n}{k}\Big )}\frac {k}{n}\frac {(r+(x+y\omega_K)n)}{\sqrt{X/a}}=u+v\omega_K,
\end{equation*}
   with $u,v \in \mr$. (If $k=0$ we omit the factors involving $k/n$.) The Jacobian of
this transformation being $aN(n)/X$ we find that
\begin{align*}
    \int\limits^{\infty}_{-\infty}\int\limits^{\infty}_{-\infty}W&\left(\frac {aN(r+(x+y\omega_K)n)}{X}\right)\widetilde{e}\left(-k(x+y\omega_K)\right)\dif x \dif y
   \\
   =& \frac {X}{aN(n)}\widetilde{e}\left(\frac {kr}{n}\right)\int\limits^{\infty}_{-\infty}\int\limits^{\infty}_{-\infty}
   W \left( N(u+v\omega_K) \right) \widetilde{e}\left (-(u+v\omega_K)\sqrt{N(k/n)X/a} \right) \dif u \dif v,
\end{align*}
   whence
\begin{align*}
    \sum_{m \in \mathcal{O}_K}\leg {m}{n}W\left(\frac {aN(m)}{X}\right) &= \frac {X}{aN(n)}\sum_{k \in
   \mathcal{O}_K}\widetilde{W}_K\left(\sqrt{\frac {N(k)X}{aN(n)}}\right)\sum_{r \shortmod n}\leg {r}{n}\widetilde{e}\left(\frac {kr}{n}\right).
\end{align*}
    The inner sum of the last expression above is $g_K(k,n)$ by definition. We apply \eqref{gtoG} to write $g_K(k,n)$ in terms of $G_K(k,n)$, using that
    $G_K(k,n)=\leg {-1}{n}G_K(-k,n)$ (by \eqref{2.7}), and recombining the $k$ and $-k$ terms to see that \eqref{PoissonsumQw} holds. This completes the proof of the lemma.
\end{proof}

    From Lemma \ref{Poissonsum}, we readily deduce the following
\begin{corollary}
\label{Poissonsumformodd} Let $n \in \mathcal{O}_K, (n, 2)=1$. For any Schwartz class function $W$,  we have
\begin{equation*}
   \sum_{\substack {c \in \mathcal{O}_K \\ (c,2)=1}}\leg {c}{n} W\left(\frac {N(c)}{X}\right)
    =\leg {c_K}{n} \frac {X}{N(n)}\sum_{\substack {(m) \\m | c_K}} \frac {\mu_K(m)}{N(m)}\sum_{\substack {k \in
   \mathcal{O}_K \\ \frac {c_K}{m} | k }}G_K(k,n)\widetilde{W}_K\left(\sqrt{\frac {N(k)X}{N(c_K)N(n)}}\right).
\end{equation*}
\end{corollary}
\begin{proof}
      By Lemma \ref{Poissonsum}, we have
\begin{equation}
\label{2.21}
\begin{split}
  \sum_{\substack {c \in \mathcal{O}_K \\ (c,2)=1}}\leg {c}{n}W \left( \frac {N(c)}{X} \right)&=\sum_{\substack {(m) \\m | c_K}}\mu_K(m)\leg {m}{n}\sum_{c}\leg
  {c}{n}W \left( \frac {N(m)N(c)}{X} \right)\\
  &=\sum_{\substack {(m) \\m | c_K}}\mu_K(m)\leg {m}{n} \frac {X}{N(m)N(n)}\sum_{k \in
   \mathcal{O}_K}G_K(k,n)\widetilde{W}_K\left(\sqrt{\frac {N(k)X}{N(m)N(n)}}\right).
   \end{split}
\end{equation}

    Using \eqref{2.7}, we have
\begin{align*}
    G_K \left( \frac {c_K}{m}k,n \right)=\leg {c_K/m}{n} G_K(k,n).
\end{align*}
    Substituting this back to last expression in \eqref{2.21}, we get the desired result.
\end{proof}

    We will require some simple estimates on $\widetilde{W}_K(t)$ and its derivatives. It is easy to see that
\begin{align*}
     \widetilde{W}_{K}(t)  =
\begin{cases}
\ \displaystyle{   \int\limits_{\mr^2}W \left( x^2+xy+y^2\frac {1-d}{4} \right)e(ty) \ \dif x \dif y}  \qquad & \text{if $d \equiv 1 \pmod 4$}, \\ \\
\ \displaystyle{    \int\limits_{\mr^2}W \left( x^2-dy^2 \right) e(ty) \ \dif x \dif y} \qquad & \text{if $d \equiv 2, 3 \pmod 4$}.
\end{cases}
\end{align*}

      We change variables in the integral, writing
\[  x+\frac {y}{2}  =x', \quad \frac {\sqrt{-d}}{2}y=y',  \; \text{if $d \equiv 1 \pmod 4$} \; \; \text{and} \; \; x=x', \quad \sqrt{-d}y=y', \text{if $d \equiv 2, 3 \pmod 4$}. \]

The Jacobian of these transformations are $\sqrt{-d}/2$ when $d \equiv 1 \pmod 4$ and $\sqrt{-d}$, otherwise.  So we find that
\begin{align*}
     \widetilde{W}_{K}(t)  =
\begin{cases}
\displaystyle{   \frac {2}{\sqrt{-d}}\int\limits_{\mr^2}W \left( x'^2+y'^2 \right) \cos \left( \frac {4\pi ty'}{\sqrt{-d}} \right) \ \dif x' \dif y' } \qquad & \text{if $d \equiv 1 \pmod 4$}, \\ \\
\displaystyle{   \frac {1}{\sqrt{-d}}\int\limits_{\mr^2}W \left( x'^2+y'^2 \right) \cos \left( \frac {4\pi ty'}{\sqrt{-d}} \right) \ \dif x' \dif y'  } \qquad & \text{if $d \equiv 2, 3 \pmod 4$}.
\end{cases}
\end{align*}

    We deduce from the above that $\widetilde{W}_K(t)\in
    \mr$ for any $t \geq 0$. Suppose that $W(t)$ is a non-negative smooth function supported on $(1,2)$, satisfying $W(t)=1$ for $t \in (1+1/U, 2-1/U)$ and $W^{(j)}(t) \ll_j U^j$ for all integers $j \geq 0$. The following estimations for $\widetilde{W}_K(t)$ and its derivatives are from \cite[Section 2.8]{G&Zhao5} (beware of
    the difference between the supports of the functions):
\begin{align}
\label{bounds}
     \widetilde{W}^{(\mu)}_K (t)\ll_{j} \min \{ 1, U^{j-1}t^{-j} \}
\end{align}
    for all integers $\mu \geq 0$, $j \geq 1$ and all $t>0$. \newline

     We also have
\begin{align}
\label{w0}
    \widetilde{W}_K(0) =A_K+O(\frac 1{U}),
\end{align}
   where
\begin{align*}
 A_K =\begin{cases}
    \displaystyle \frac {2\pi}{\sqrt{-d}}  \qquad & \text{if $d \equiv 1 \pmod 4$ }, \\ \\
    \displaystyle \frac {\pi}{\sqrt{-d}}  \qquad & \text{if $d \equiv 2, 3 \pmod 4$ }.
\end{cases}
\end{align*}

    As $\Phi(t)$ satisfies the assumptions on $W(t)$, the estimations \eqref{bounds}-\eqref{w0} are valid for $\widetilde{\Phi}_K(t)$.  In what
    follows, we shall hence use these estimations for $\widetilde{\Phi}_K(t)$ without further justification.

\section{Proof of Theorem \ref{quadraticmainthm}}
\label{Section 3}

\subsection{Evaluation of  $C_K(X)$}

     We have
\begin{align*}
    \sumstar_{\substack {N(c) \leq X \\ (c, 2)=1}}1=\sum_{\substack {N(c) \leq X \\ (c, 2)=1}}\mu^2_{K}(c)=\sum_{\substack {N(c) \leq X \\ (c,
    2)=1}}\sum_{\substack { (l) \\ l^2 | c \\ (l,2)=1}}\mu_{K}(l)
    =\sum_{\substack {(l) \\ N(l) \leq \sqrt{X} \\ (l,2)=1}}\mu_{K}(l) \sum_{\substack {N(c) \leq X/N(l^2) \\ (c, 2)=1}} 1,
\end{align*}
    where the ``$*$'' on the sum over $c$ means that the sum is restricted to square-free elements $c$ of $\mathcal{O}_K$. \newline

   Using analyses similar to those in the study of the Gauss circle problem, we have, with $\theta = 131/416$ (see \cite{Huxley1}), that
\begin{align*}
  \sum_{N(a) \leq x} 1 =A_K x+O(x^{\theta}).
\end{align*}

    This implies that
\begin{align*}
     \sum_{\substack {N(c) \leq X/N(l^2) \\ (c, 2)=1}} 1 =
     A_K\left (\sum_{\substack {(m) \\m | c_K}} \frac {\mu_K(m)}{N(m)} \right )\frac {X}{N(l^2)}+O\left( \left( \frac {X}{N(l^2)} \right)^{\theta} \right).
\end{align*}

    We then deduce that
\begin{align*}
    \sumstar_{\substack {N(c) \leq X \\ (c, 2)=1}}1 =
     \displaystyle  X A_K \sum_{\substack {(m) \\m | c_K}} \frac {\mu_K(m)}{N(m)}\sum_{\substack { (l) \\   (l,2)=1 }}\frac
     {\mu_{K}(l)}{N(l^2)} +O(X^{1/2}).
\end{align*}

   We conclude from this that as $X \rightarrow \infty$,
\begin{align*}
    \sumstar_{(c, 2)=1}\Phi \left( \frac {N(c)}{X} \right) \sim \# C_K(X) \sim \displaystyle  X A_K \sum_{\substack {(m) \\m | c_K}} \frac {\mu_K(m)}{N(m)}\sum_{\substack { (l) \\ (l,2)=1 }}\frac
  {\mu_{K}(l)}{N(l^2)}.
\end{align*}

    Using arguments similar to those in \cite[Section 3.1]{G&Zhao4} and Lemma \ref{lem2.1},  we see that Theorem \ref{quadraticmainthm} follows from
\begin{equation}
\label{01.50}
  \lim_{X \rightarrow \infty} \frac{S_K(X, Y;\hat{\phi}, \Phi)}{X \log X}=
  \displaystyle -\frac {A_K }{4}\sum_{\substack {(m) \\m | c_K}} \frac {\mu_K(m)}{N(m)}\sum_{\substack { (l) \\   (l,2)=1 }}\frac
  {\mu_{K}(l)}{N(l^2)}\int\limits^{\infty}_{-\infty} \left( 1-\chi_{[-1,1]}(t) \right) \hat{\phi}(t) \dif t,
\end{equation}
   where $\phi$ is any Schwartz function with $\hat{\phi}$ supported in $(-2+\varepsilon, 2-\varepsilon)$ for any $0<\varepsilon<1$ and
\begin{align*}
    S_K(X,Y; \hat{\phi}, \Phi) =
    \sumstar_{(c, 2)=1} \ \sum_{\substack{ \varpi \text{ primary} \\ N(\varpi) \leq Y}} \frac {\chi^{(-4c_Kc)}(\varpi)\log
    N(\varpi)}{\sqrt{N(\varpi)}}\hat{\phi} \left( \frac {\log N(
   \varpi)}{\log X} \right) \Phi \left( \frac {N(c)}{X} \right).
\end{align*}
    Here $Y=X^{2-\epsilon}$ and we shall write the condition $N(\varpi) \leq Y$ explicitly throughout this section.
\subsection{Expressions $S_{K,M}(X,Y; \hat{\phi}, \Phi)$ and $S_{K,R}(X,Y; \hat{\phi}, \Phi)$ }
     Let $Z=\log^5 X $ and write
     $\mu_{K}^2(c)=M_{K,Z}(c)+R_{K,Z}(c)$ where
\begin{equation*}
    M_{K,Z}(c)=\sum_{\substack {(l), \ l^2|c \\ N(l) \leq Z}}\mu_{K}(l) \; \quad \mbox{and} \; \quad  R_{K,Z}(c)=\sum_{\substack {(l), \ l^2|c \\ N(l) >
    Z}}\mu_{K}(l).
\end{equation*}

   Define
\[ S_{K,M}(X,Y; \hat{\phi}, \Phi) =\sum_{(c, 2)=1}M_{K,Z}(c) \sum_{\substack{ \varpi \text { primary} \\ N(\pi) \leq Y}} \frac {\log
N(\varpi)}{\sqrt{N(\varpi)}}\chi^{(-4c_Kc)} (\varpi)  \hat{\phi} \left( \frac {\log N(
   \varpi)}{\log X} \right)\Phi \left( \frac {N(c)}{X} \right), \]
    and
\[ S_{K,R}(X,Y; \hat{\phi}, \Phi) =\sum_{(c, 2)=1}R_{K,Z}(c) \sum_{\substack{ \varpi \text { primary} \\ N(\pi) \leq Y}} \frac {\log
N(\varpi)}{\sqrt{N(\varpi)}}\chi^{(-4c_Kc)}(\varpi)  \hat{\phi} \left( \frac {\log N(
   \varpi)}{\log X} \right) \Phi \left( \frac {N(c)}{X} \right), \]
so that $S_K(X,Y; \hat{\phi}, \Phi)=S_{K,M}(X,Y; \hat{\phi}, \Phi)+S_{K,R}(X,Y; \hat{\phi}, \Phi)$. \newline

   As the proof of Theorem \ref{quadraticmainthm} is similar to those of \cite[Theorem 1.1]{G&Zhao4}. We shall focus on obtaining the (secondary) main term and be brief on the error terms in what follows. \newline

   We write $S_{K,M}(X,Y; \hat{\phi},\Phi)$ as
\begin{align*}
     S_{K,M}(X,Y; \hat{\phi}, \Phi)   = \sum_{\varpi \text { primary}} \frac {\log N(\varpi)}{\sqrt{N(\varpi)}} \leg{-c_K}{\varpi} \hat{\phi} \left( \frac {\log
     N(
   \varpi)}{\log X} \right) \sum_{\substack {(l) \\ (l,2)=1 \\ N(l) \leq Z}} \mu_{K}(l)\leg {l^2}{\varpi}  \sum_{\substack {c \in \mathcal{O}_K \\ (c, 2)=1} }
   \leg {c}{\varpi}\Phi \left( \frac {N(cl^2)}{X} \right).
\end{align*}

    Applying Corollary \ref{Poissonsumformodd}, we obtain that
\begin{align*}
    \sum_{\substack {c \in \mathcal{O}_K \\ (c, 2)=1} } \leg {c}{\varpi}\Phi \left( \frac {N(cl^2)}{X} \right)
   = \leg {c_K}{\varpi} \frac {X}{N(l^2\varpi)}\sum_{\substack {(m) \\m | c_K}} \frac {\mu_K(m)}{N(m)}\sum_{\substack {k \in
   \mathcal{O}_K \\ \frac {c_K}{m} | k }}G_K(k,\varpi)\widetilde{\Phi}_K\left(\sqrt{\frac {N(k)X}{N(c_Kl^2\varpi)}}\right).
\end{align*}

   We can now recast $S_{K,M}(X,Y; \hat{\phi}, \Phi)$ as
\begin{align*}
    & S_{K,M}(X,Y; \hat{\phi}, \Phi)  \\
    =& X\sum_{\substack {(l) \\ (l,2)=1 \\ N(l) \leq Z}} \frac {\mu_{K}(l)}{N(l^2)} \sum_{\substack {(m) \\m | c_K}} \frac {\mu_K(m)}{N(m)}
    \sum_{\substack{k \in
   \mathcal{O}_K \\ \frac {c_K}m |k }}\sum_{\varpi \text { primary}} \frac {\log N(\varpi)}{N(\varpi)^{3/2}}\leg {-l^2}{\varpi} G_K(k,
   \varpi)\hat{\phi} \left( \frac {\log N(\varpi)}{\log X} \right) \widetilde{\Phi}_K\left(\sqrt{\frac {N(k)X}{N(c_Kl^2\varpi)}}\right) \nonumber \\
   =& X\sum_{\substack {(l) \\ (l,2)=1 \\ N(l) \leq Z}} \frac {\mu_{K}(l)}{N(l^2)} \sum_{\substack {(m) \\m | c_K}} \frac {\mu_K(m)}{N(m)} \sum_{\substack
   {k \in
   \mathcal{O}_K \\ \frac {c_K}m |k \\ k \neq 0}}\sum_{\varpi \text { primary}} \frac {\log N(\varpi)}{N(\varpi)}\leg
   {kl^2}{\varpi}\hat{\phi} \left( \frac {\log N(\varpi)}{\log X} \right) \widetilde{\Phi}_K\left(\sqrt{\frac {N(k)X}{N(c_Kl^2\varpi)}}\right), \nonumber
\end{align*}
   as one deduces easily from Lemma \ref{Gausssum} that $G_K(0, \varpi)=0$ and when $k \neq 0$,
\begin{align*}
    G_K(k, \varpi)=\leg {-k}{\varpi}N(\varpi)^{1/2}.
\end{align*}

\subsection{Estimation of $S_{K,M}(X,Y; {\hat \phi}, \Phi)$,  the second main term}
   We define $\chi^{(kl^)}$ to be $\leg {kl^2}{\cdot}$. Similar to our discussions in Section \ref{sec2.4}, when $k \neq 0$ is not a square, $\chi_{kl^2}$ can be regarded as a non-principle Hecke character modulo $4kl^2$ of trivial infinite type. Just as the estimations in \cite[Section 3.6]{G&Zhao4}, we deduce via Lemma \ref{lem2.7} that the contribution of $k \neq 0, \square$ in $S_{K,M}(X,Y; {\hat \phi}, \Phi)$ is $o(X \log X)$. \newline
   
   We denote the contribution of the terms $k=\square, k \neq 0$ in $S_{K,M}(X,Y; \hat{\phi}, \Phi)$ by $S_{K,M, \square}(X,Y; \hat{\phi}, \Phi)$. Thus
\begin{align}
\label{3.2}
    S_{K,M}(X,Y; \hat{\phi}, \Phi) =S_{K,M, \square}(X,Y; \hat{\phi}, \Phi)+o(X \log X).
\end{align}   

   In what follows, we show that $S_{K,M, \square}(X,Y; \hat{\phi}, \Phi)$ gives rise to a second main term. Before we proceed, we need the following
    result:
\begin{lemma}
\label{Poissonsumoverk} For $y>0$,
\begin{align*}
    \sum_{\substack {(m) \\m | c_K}} \frac {\mu_K(m)}{N(m)}\sum_{\substack {k \in
   \mathcal{O}_K \\  \frac {c_K}{m} | k \\ k \neq 0 }}\widetilde{\Phi}_K\left(\frac {N(k)}{y}\right)
   =- \widetilde{\Phi}_K\left(0\right)\sum_{\substack {(m) \\m | c_K}} \frac {\mu_K(m)}{N(m)}+O \left( \frac {U^2}{y^{1/2}} \right).
\end{align*}
\end{lemma}
\begin{proof}
   Note that
\begin{align*}
   \sum_{\substack {(m) \\m | c_K}} \frac {\mu_K(m)}{N(m)}\sum_{\substack {k \in
   \mathcal{O}_K \\  \frac {c_K}{m} | k \\ k \neq 0}}\widetilde{\Phi}_K\left(\frac {N(k)}{y}\right)
   =\sum_{\substack {(m) \\m | c_K}} \frac {\mu_K(m)}{N(m)}\sum_{\substack {k \in
   \mathcal{O}_K \\ \frac {c_K}{m} | k }}\widetilde{\Phi}_K\left(\frac {N(k)}{y}\right)- \widetilde{\Phi}_K\left(0\right)\sum_{\substack {(m) \\m | c_K}}
   \frac {\mu_K(m)}{N(m)}
\end{align*}
and
\begin{align*}
    \sum_{\substack {(m) \\m | c_K}} \frac {\mu_K(m)}{N(m)}\sum_{\substack {k \in
   \mathcal{O}_K \\ \frac {c_K}{m} | k }}\widetilde{\Phi}_K\left(\frac {N(k)}{y}\right)= \sum_{\substack {(m) \\m | c_K}} \frac
   {\mu_K(m)}{N(m)}\sum_{\substack {k \in
   \mathcal{O}_K}}\widetilde{\Phi}_K\left(\frac {N(c_K/m)N(k)}{y}\right).
\end{align*}
   By taking $n=1$ in \eqref{PoissonsumQw}, we immediately obtain that for $m | c_K$,
\begin{equation*}
      \sum_{k \in \mathcal{O}_K }\widetilde{\Phi}\left(\frac {N(c_K/m)N(k)}{y}\right) = \frac {y}{N(c_K/m)}\sum_{j \in
   \mathcal{O}_K }\breve{\Phi}\left(\sqrt{\frac {N(j)y}{N(c_K/m)}}\right),
\end{equation*}
   where
\begin{align*}
   \breve{\Phi}(t) =\int\limits^{\infty}_{-\infty}\int\limits^{\infty}_{-\infty}\widetilde{\Phi}(N(u+v\omega_K))\widetilde{e}\left(-
   t(u+v\omega_K)\right)\dif u \dif v, \quad t \geq 0.
\end{align*}

    It follows that
\begin{align}
\label{2.200}
    \sum_{\substack {(m) \\m | c_K}} \frac {\mu_K(m)}{N(m)}\sum_{\substack {k \in
   \mathcal{O}_K}}\widetilde{\Phi}_K\left(\frac {N(c_K/m)N(k)}{y}\right)= \frac {y}{N(c_K)}\sum_{\substack {(m) \\m | c_K}} \mu_K(m)\sum_{\substack{ j \in
   \mathcal{O}_K}}\breve{\Phi}\left(\sqrt{\frac {N(j)y}{N(c_K/m)}}\right).
\end{align}
    We have, when $t>0$, using \eqref{bounds} and via integration by parts
\begin{align*}
     \breve{\Phi}(t) \ll \frac {U^2}{t^3}.
\end{align*}
    The lemma follows immediately from this bound and \eqref{2.200}.
\end{proof}

    Now, by a change of variables $k \mapsto k^2$ and noting that $k^2=k_1^2$ if and only if $k =\pm k_1$, we see that
\begin{align*}
   & S_{K,M, \square}(X,Y; \hat{\phi}, \Phi) \\
   = & \frac {X}{2} \sum_{\substack {(l) \\ (l,2)=1 \\ N(l) \leq Z}} \frac {\mu_{K}(l)}{N(l^2)} \sum_{\substack {(\varpi, l)=1 \\ \varpi \text{ primary}}}
   \frac {\log N(\varpi)}{N(\varpi)}\hat{\phi} \left( \frac {\log N(
   \varpi)}{\log X} \right) \sum_{\substack {(m) \\m | c_K}} \frac {\mu_K(m)}{N(m)} \sum_{\substack {k \in
   \mathcal{O}_K \\ \frac {c_K}m |k \\ k \neq 0 \\  (k, \varpi)=1 }}\widetilde{\Phi}_K\left(N(k)\sqrt{\frac {X}{N(c_Kl^2\varpi)}}\right) \\
   =&\frac {X}{2} \sum_{\substack {(l) \\  (l,2)=1 \\ N(l) \leq Z}} \frac {\mu_{K}(l)}{N(l^2)}\sum_{\substack {(\varpi, l)=1 \\ \varpi
\text{ primary} \\ N(\varpi) \geq X/N(l^2) }} \frac {\log N(\varpi)}{N(\varpi)}\hat{\phi} \left( \frac {\log N(
   \varpi)}{\log X} \right) \sum_{\substack {(m) \\m | c_K}} \frac {\mu_K(m)}{N(m)} \sum_{\substack {k \in
   \mathcal{O}_K \\ \frac {c_K}m |k \\ k \neq 0}}\widetilde{\Phi}_K\left(N(k)\sqrt{\frac {X}{N(c_Kl^2\varpi)}}\right)+o(X\log X) \\
   =& -\widetilde{\Phi}(0)\frac {X}{2}\sum_{\substack {(m) \\m | c_K}} \frac {\mu_K(m)}{N(m)}\sum_{\substack {(l) \\ (l,2)=1 \\ N(l) \leq Z
}}  \frac {\mu_{K}(l)}{N(l^2)} \sum_{\substack {(\varpi, l)=1 \\ \varpi \text{ primary}  \\ N(\varpi) \geq X/N(l^2) }} \frac {\log
N(\varpi)}{N(\varpi)}\hat{\phi} \left( \frac {\log N(
   \varpi)}{\log X} \right)+o(X\log X). 
\end{align*}
   where the second equality above follows by noting that removing the condition $(k, \varpi)=1$ and then restricting the sum over $\varpi$ to those satisfying $N(\varpi) \geq X/N(l^2)$ only gives rise to an error term of $o(X\log X)$, the last equality above follows from Lemma \ref{Poissonsumoverk}. \newline

    Applying the following consequence of the prime ideal theorem \cite[Theorem 8.9]{MVa1} which asserts that for $x \geq 1$,
\begin{equation*}
    \sum_{\substack{ N(\varpi) \leq x \\ \varpi \text{ primary}}} \frac {\log N(\varpi)}{N(\varpi) }= \log x+O(\log \log 3x),
\end{equation*}
    we deduce via partial summation and \eqref{w0},
\begin{align}
\label{SKM}
    S_{K,M, \square}(X,Y; \hat{\phi}, \Phi) =-\frac {X \log X}{4}A_K \sum_{\substack {(m) \\m | c_K}} \frac {\mu_K(m)}{N(m)} \sum_{\substack {(l) \\ (l,2)=1}}  \frac
{\mu_{K}(l)}{N(l^2)}\int\limits^{\infty}_{-\infty} \left( 1-\chi_{[-1,1]}(t) \right) \hat{\phi}(t) \dif t+o(X \log X).
\end{align}
   
\subsection{Conclusion }  
    We estimate $S_{K,R}(X,Y; \hat{\phi}, \Phi)$ as in \cite[Section 3.3]{G&Zhao4} to see that
\begin{equation*}
    S_{K,R}(X,Y; \hat{\phi}, \Phi)=o(X\log X).
\end{equation*} 
   We now combine the above estimation with \eqref{3.2} and \eqref{SKM} to obtain that
\begin{align*}
  S_K(X, Y;\hat{\phi}, \Phi )
& = -\frac {X \log X}{4}A_K \sum_{\substack {(m) \\m | c_K}} \frac {\mu_K(m)}{N(m)} \sum_{\substack {(l) \\ (l,2)=1}}  \frac
{\mu_{K}(l)}{N(l^2)}\int\limits^{\infty}_{-\infty} \left( 1-\chi_{[-1,1]}(t) \right) \hat{\phi}(t) \dif t+o\left( X\log X \right),
\end{align*}
 which implies \eqref{01.50} and this completes the proof of Theorem \ref{quadraticmainthm}. \newline

\noindent{\bf Acknowledgments.} P. G. is supported in part by NSFC grant 11871082 and L. Z. by the FRG grant PS43707.  The authors thank H. M. Bui for some helpful discussions and the anonymous referee for his/her comments and suggestions.

\bibliography{biblio}
\bibliographystyle{amsxport}

\vspace*{.5cm}

\noindent\begin{tabular}{p{8cm}p{8cm}}
School of Mathematical Sciences & School of Mathematics and Statistics \\
Beihang University & University of New South Wales \\
Beijing 100191 China & Sydney NSW 2052 Austrlia \\
Email: {\tt penggao@buaa.edu.cn} & Email: {\tt l.zhao@unsw.edu.au} \\
\end{tabular}

\end{document}